\documentclass[reqno,12pt]{amsart}

\usepackage{amsmath,amsthm,amsfonts,amssymb}
\usepackage{enumerate}

\newtheorem{theorem}{Theorem}[section]

\newtheorem{lemma}[theorem]{Lemma}
\newtheorem{corollary}[theorem]{Corollary}

\newtheorem{remark}[theorem]{Remark}

\begin{document}

\centerline{\bfseries Some remarks on Talagrand's convolution conjecture}
\vskip .2in
\centerline{\bfseries Alexander Shaposhnikov\footnotemark}
\footnotetext{e-mail: shal1t7@mail.ru}
\vskip .2in

\centerline{{\bf Abstract}}

\vskip .1in

We present a new martingale--based approach to Talagrand's convolution conjecture in the Gaussian and Boolean cube settings.  The argument yields explicit constants in the weak--type estimates. 

\section{Introduction}

Let $\mu$ be the uniform probability measure on $\{-1,1\}^n$, and let
$T_\rho$, $0<\rho<1$, be the noise operator. Talagrand's convolution
conjecture \cite{Talagrand} asserts that
\[
\mu\left(T_\rho f>\eta\int f\,d\mu\right)
\leq\frac{C_\rho}{\eta\sqrt{\log\eta}},\qquad \eta>1,
\]
for every $f\geq0$, where $C_\rho$ depends only on $\rho$.

The Gaussian counterpart, studied in \cite{BallEtAl}, replaces $\mu$
by the standard Gaussian measure $\gamma_n$ on $\mathbb{R}^n$ and
$T_\rho$ by the Ornstein--Uhlenbeck operator
\[
Q_\rho f(x):=\int_{\mathbb{R}^n}
f(\rho x+\sqrt{1-\rho^2}y)\,d\gamma_n(y).
\]
Eldan and Lee \cite{EldanLee} proved the Gaussian bound with an additional
iterated-logarithmic factor, which was removed by Lehec \cite{Lehec}.

The Boolean conjecture was recently settled by Lu, Guo and Fang
\cite{LuGuoFang} using an AI-generated proof based on Chen's perturbed
reverse-heat approach \cite{Chen},  see also the refinement in \cite{XiangZhang}.
In this note we give martingale proofs in both settings. Our arguments do
not use the coupling constructions of the above works or semilog-convexity
estimates. We obtain the explicit constants $[2(1-\rho^2)]^{-1/2}$ and
$\sqrt{(1+\rho)/(1-\rho)}$ in the Gaussian and Boolean cases, respectively.
Both are of order $(1-\rho)^{-1/2}$ as $\rho\uparrow1$, or equivalently
$t^{-1/2}$ as $t\downarrow0$ when $\rho=e^{-t}$. In the Boolean case this
improves the order $(1-\rho)^{-3/2}\sqrt{\log(1/(1-\rho))}$ in
\cite[Theorem~1.1]{LuGuoFang}. We refer the reader to the above papers for
a detailed exposition of the history of the problem.

\section{Gaussian case}

Let $B$ be the standard Brownian motion in $\mathbb{R}^n$ with the standard
filtration $(\mathcal{F}_t)_{t\geq0}$ and law $\mathbb{P}$. Let
$(P_t)_{t\geq0}$ be the standard heat semigroup on $\mathbb{R}^n$.
For a Borel-measurable function $f\geq0$, put $F_1:=f(B_1)$ and assume
$\mathbb{E}F_1=1$. Set
\[
F_t:=P_{1-t}f(B_t)=\mathbb{E}[F_1\mid\mathcal{F}_t],
\qquad d\mathbb{Q}:=F_1\,d\mathbb{P}.
\]

Fix $0<s<1$ and $\eta>1$. For now, suppose that $f$ is continuously
differentiable with bounded gradient, bounded above and bounded away
from zero. It\^o's formula gives
\[
d\log F_t
=\frac{\langle\nabla P_{1-t}f(B_t),dB_t\rangle}{F_t}
-\frac{|\nabla P_{1-t}f(B_t)|^2}{2F_t^2}\,dt.
\]
Set
\[
A:=\{F_s>\eta\},
\qquad
\sigma:=\inf\{t\in[0,s]:F_t\geq\eta\}\wedge s,
\]
and define, for $0\leq t\leq s$,
\[
\begin{gathered}
u_t:=1_{\{t\leq\sigma\}}
\frac{\nabla P_{1-t}f(B_t)}{F_t},\\[1mm]
M_t:=\int_0^t\langle u_r,dB_r\rangle,
\qquad [M]_t:=\int_0^t|u_r|^2\,dr,
\qquad U_t:=\int_0^t u_r\,dr.
\end{gathered}
\]
Then
\[
\begin{gathered}
\log F_{t\wedge\sigma}=M_t-\tfrac12[M]_t,\\[1mm]
M_s=\log F_\sigma+\tfrac12[M]_s\geq\log\eta\quad\text{on }A.
\end{gathered}
\]

\emph{Plan of the proof.} The main observation, proved in
Lemma~\ref{lem:representation}, is that
\[
(F_1-\eta)1_A-\mathbb{E}(F_s-\eta)_+
=\int_0^1\langle h_t,dB_t\rangle,
\qquad h\text{ is a martingale}.
\]
We extend $u$ to $(s,1]$ by $-U_s/(1-s)$, so that $U_1=0$.
Since $\mathbb{E}M_1=0$, It\^o's isometry and the martingale property of
$h$ then give
\[
\mathbb{E}[(F_1-\eta)1_AM_1]
=\mathbb{E}\langle U_1,h_1\rangle=0.
\]
Thus, using $\mathbb{E}[M_1\mid\mathcal{F}_s]=M_s$,
\[
\mathbb{E}[F_1 1_AM_1]
=\eta\,\mathbb{E}[1_AM_s]
\geq\eta\log\eta\,\mathbb{P}(A).
\]
The same condition $U_1=0$ gives the weighted It\^o isometry:
\[
\mathbb{E}[F_1M_1]=0,
\qquad
\mathbb{E}[F_1M_1^2]=\mathbb{E}[F_1[M]_1]
\leq\frac{2\log\eta}{1-s}.
\]
Here stopping gives $\mathbb{E}[F_1[M]_s]\leq2\log\eta$, and the return
costs at most a factor $(1-s)^{-1}$ in energy. Cauchy--Schwarz therefore
bounds the same expectation from above:
\[
\begin{aligned}
\mathbb{E}[F_1 1_AM_1]
&=\mathbb{E}[F_1(1_A-\tfrac12)M_1]\\
&\leq\tfrac12\sqrt{\mathbb{E}[F_1M_1^2]}
\leq\sqrt{\frac{\log\eta}{2(1-s)}}.
\end{aligned}
\]
\pagebreak[3]

\begin{lemma}\label{lem:representation}
Suppose that $f$ is bounded and continuously differentiable, with bounded
gradient. For any $\eta>0$ and $0<s<1$, there exists a martingale
$(h_t)_{0\leq t\leq1}$ such that
\[
1_A(F_t-\eta)-\mathbb{E}(F_s-\eta)_+
=\int_0^t\langle h_r,dB_r\rangle,
\qquad s\leq t\leq1,
\]
where $A=\{F_s>\eta\}$.
\end{lemma}

\begin{proof}
Let us recall that
\[
dF_t=\langle\nabla P_{1-t}f(B_t),dB_t\rangle.
\]
Therefore, for $s<t<1$,
\[
1_A(F_t-\eta)-(F_s-\eta)_+
=\int_s^t\langle 1_A\nabla P_{1-r}f(B_r),dB_r\rangle.
\]
At the same time, It\^o's formula applied to the backward heat extension
of $(P_{1-s}f-\eta)_+$, first before time $s$ and then by an $L^2$ limit, gives
\[
(F_s-\eta)_+-\mathbb{E}(F_s-\eta)_+
=\int_0^s\langle h_r,dB_r\rangle,
\]
where
\[
h_r:=\nabla P_{s-r}[(P_{1-s}f-\eta)_+](B_r),\qquad r<s.
\]
The almost-everywhere chain rule for the positive part and commutation
of the gradient with the heat semigroup give
\[
h_r=\mathbb{E}[1_A\nabla P_{1-s}f(B_s)\mid\mathcal{F}_r],
\qquad r<s.
\]
Thus $(h_r)$ is a martingale on $[0,s]$ with
\[
h_s=1_A\nabla P_{1-s}f(B_s).
\]
Since $A\in\mathcal{F}_s$, extending it by
$h_r=1_A\nabla P_{1-r}f(B_r)$ for $s<r<1$, with
$h_1=1_A\nabla f(B_1)$, gives a martingale on $[0,1]$.
Adding the two integral identities completes the proof.
\end{proof}

\begin{lemma}\label{lem:orthogonality}
Let $u$ be a bounded, progressively measurable $\mathbb{R}^n$-valued
process on $[0,1]$ and assume that
\[
M_t=\int_0^t\langle u_r,dB_r\rangle,
\qquad \int_0^1u_r\,dr=0\quad\text{almost surely}.
\]
Then for any square-integrable stochastic integral
\[
N_1=\int_0^1\langle v_r,dB_r\rangle
\]
such that $(v_r)_{0\leq r\leq1}$ is a martingale, the following equality holds:
\[
\mathbb{E}[M_1N_1]=0.
\]
\end{lemma}

\begin{proof}
Indeed, by It\^o's isometry and the martingale property of $v$,
\[
\begin{aligned}
\mathbb{E}[M_1N_1]
&=\mathbb{E}\int_0^1\langle u_r,v_r\rangle\,dr
=\mathbb{E}\int_0^1\langle u_r,v_1\rangle\,dr\\
&=\mathbb{E}\left\langle\int_0^1u_r\,dr,v_1\right\rangle=0.
\end{aligned}
\]
\end{proof}

Combining Lemmas~\ref{lem:representation} and~\ref{lem:orthogonality},
one readily obtains, for such $f$ and $u$,
\[
\mathbb{E}\left[
M_1\bigl((F_1-\eta)1_A-\mathbb{E}(F_s-\eta)_+\bigr)
\right]=0.
\]
Equivalently, taking into account $\mathbb{E}M_1=0$ and
$A\in\mathcal{F}_s$,
\begin{equation}\label{eq:comparison}
\mathbb{E}_{\mathbb{Q}}[1_AM_1]
=\mathbb{E}[M_1F_1 1_A]
=\eta\mathbb{E}[1_AM_1]
=\eta\mathbb{E}[1_AM_s].
\end{equation}

We next use the following simple observation that we were unable to
find in the literature.

\begin{lemma}\label{lem:girsanov}
Let $u$ be a bounded, progressively measurable $\mathbb{R}^n$-valued
process on $[0,1]$ with respect to $(\mathcal{F}_t)_{t\geq0}$. Let
\[
M_t:=\int_0^t\langle u_r,dB_r\rangle,
\qquad [M]_t:=\int_0^t|u_r|^2\,dr,
\qquad U_t:=\int_0^t u_r\,dr.
\]
If $U_1=0$ almost surely, then for every $\lambda\in\mathbb{R}$,
\[
\mathbb{E}\left[
 e^{\lambda M_1-\lambda^2[M]_1/2}\,\middle|\,B_1
\right]=1.
\]
\end{lemma}

\begin{proof}
Set
\[
Z_t:=\exp\left(\lambda M_t-\frac{\lambda^2}{2}[M]_t\right).
\]
Since $u$ is bounded, $Z$ is a true martingale. By the classical
Girsanov theorem, the process
\[
W_t=B_t-\lambda\int_0^t u_r\,dr
\]
is Brownian motion under the measure $Z_1\,d\mathbb{P}$ with respect to
the filtration $(\mathcal{F}_t)_{0\leq t\leq1}$. For any bounded
measurable $g$, we have
\[
\mathbb{E}[g(B_1)]
=\mathbb{E}[g(W_1)Z_1]
=\mathbb{E}[g(B_1)Z_1],
\]
where the second equality follows from the assumption $U_1=0$.
Since this holds for all bounded measurable $g$, we obtain the desired
conclusion.
\end{proof}

\begin{corollary}\label{cor:moments}
Under the assumptions of Lemma~\ref{lem:girsanov},
\[
\mathbb{E}[M_1\mid B_1]=0,
\qquad
\mathbb{E}[M_1^2\mid B_1]=\mathbb{E}[[M]_1\mid B_1].
\]
\end{corollary}

\begin{proof}
This follows by differentiating the identity in
Lemma~\ref{lem:girsanov} once and twice at $\lambda=0$.
Boundedness of $u$ justifies differentiation.
\end{proof}

This second-moment identity
\[
\mathbb{E}[M_1^2\mid B_1]=\mathbb{E}[[M]_1\mid B_1]
\]
will play a crucial role in the proof of the main result.

\begin{theorem}
For every Borel-measurable $f\geq0$ with $\mathbb{E}F_1=1$,
every $0<s<1$, and every $\eta>1$,
\[
\mathbb{P}\bigl(P_{1-s}f(B_s)>\eta\bigr)
\leq\frac{1}{\eta\sqrt{2(1-s)\log\eta}}.
\]
\end{theorem}

\begin{proof}
First suppose that $f$ is as in the proof outline. Use the stopped drift
$u$, the processes $M,U$, and $A,\sigma$ defined at the start of this
section. By continuity, $F_\sigma\leq\eta$ and $F_\sigma=\eta$ on $A$.
Conditioning on $\mathcal{F}_t$ and applying It\^o's formula
to $F\log F$, as in Shiryaev's proof of the logarithmic Sobolev
inequality \cite[Section~4]{Shiryaev}, yields
\[
\mathbb{E}_{\mathbb{Q}}[M]_s
=\mathbb{E}\int_0^\sigma\frac{d[F]_t}{F_t}
=2\mathbb{E}[F_\sigma\log F_\sigma]
\leq2\log\eta,
\]
where the last inequality uses $F_\sigma\leq\eta$ and
$\mathbb{E}F_\sigma=1$.

For $s<t\leq1$, set $u_t:=-U_s/(1-s)$ and continue $M$ and $U$
by their integral definitions. Then $u$ remains bounded,
$U_1=0$, and
\[
[M]_1=[M]_s+\frac{|U_s|^2}{1-s}
\leq\frac{[M]_s}{1-s},
\]
since $|U_s|^2\leq s[M]_s$. By Corollary~\ref{cor:moments},
\[
\mathbb{E}_{\mathbb{Q}}M_1=0,
\qquad
\mathbb{E}_{\mathbb{Q}}M_1^2
=\mathbb{E}_{\mathbb{Q}}[M]_1
\leq\frac{2\log\eta}{1-s}.
\]
Finally, \eqref{eq:comparison} and Cauchy--Schwarz give
\[
\begin{aligned}
\eta\log\eta\,\mathbb{P}(A)
&\leq\eta\mathbb{E}[1_AM_s]
=\mathbb{E}_{\mathbb{Q}}[1_AM_1]\\
&=\mathbb{E}_{\mathbb{Q}}[(1_A-\tfrac12)M_1]
\leq\tfrac12\sqrt{\mathbb{E}_{\mathbb{Q}}M_1^2}
\leq\sqrt{\frac{\log\eta}{2(1-s)}}.
\end{aligned}
\]
Dividing by $\eta\log\eta$ proves the result for such $f$.

For general $f$, choose smooth $f_k$ bounded above and away from zero,
with bounded gradient, $\mathbb{E}f_k(B_1)=1$, and
$\mathbb{E}|f_k(B_1)-f(B_1)|\to0$. Conditional expectation is an
$L^1$-contraction, so $P_{1-s}f_k(B_s)\to P_{1-s}f(B_s)$ in $L^1$.
Passing to an almost surely convergent subsequence and applying Fatou's
lemma to the indicators of $\{P_{1-s}f_k(B_s)>\eta\}$ proves the result.
\end{proof}

Since $Q_\rho f(x)=P_{1-\rho^2}f(\rho x)$ and
$B_{\rho^2}\stackrel{d}{=}\rho G$ for $G\sim\gamma_n$, taking $s=\rho^2$
gives the Ornstein--Uhlenbeck formulation.

\begin{remark}
Let $f\geq0$, $\int f\,d\gamma_n=1$, and suppose that
$K:=W_1(f\,d\gamma_n,\gamma_n)<\infty$, where $W_1$ is the
$L^1$ Kantorovich distance with respect to the Cameron--Martin norm. By
\cite[Lemma~2.2]{BogachevShaposhnikovShaposhnikov},
\[
\int_{\mathbb{R}^n} Q_{e^{-t}}f\,
\sqrt{\log(1+Q_{e^{-t}}f)}\,d\gamma_n
\leq\sqrt{\log2}+\frac{K}{2\sqrt{t}},\qquad 0<t\leq1.
\]
In particular, by the semigroup property and Jensen's inequality,
$Q_\rho f\in L\log^{1/2}L(\gamma_n)$ for every $0<\rho<1$.
\end{remark}

\section{Boolean case}

Let $\mu$ be the uniform probability measure on $\{-1,1\}^n$ and let
$0<\rho<1$. For $f\geq0$ with $\mathbb{E}_{\mu}f=1$, set
\[
F(x):=T_{\rho}f(x)
=
2^{-n}\sum_{z\in\{-1,1\}^n}
f(z)\prod_{i=1}^n(1+\rho x_i z_i).
\]
Under $\mathbb{Q}$, let $Z$ have distribution $f\,d\mu$ and,
conditionally on $Z$, let $X_1,\ldots,X_n$ be independent with
\[
\mathbb{Q}(X_i=x_i\mid Z)=\frac{1+\rho Z_i x_i}{2},
\qquad x_i\in\{-1,1\}.
\]
Then, for every function $g$ on the cube,
\[
\mathbb{E}_{\mathbb{Q}}g(X)=\mathbb{E}_{\mu}[Fg].
\]

Let $\mathcal{F}_i:=\sigma(X_1,\ldots,X_i)$, and let $x^{(i)}$
denote $x$ with its $i$-th coordinate flipped. Fix $\eta\geq e$ and set
\[
F_k:=\mathbb{E}_{\mu}[F\mid\mathcal{F}_k],
\qquad A:=\{F>\eta\},
\qquad
\sigma:=\min\bigl(\{k\leq n:F_k\geq\eta\}\cup\{n\}\bigr).
\]
Define
\[
u_i:=1_{\{i\leq\sigma\}}
\mathbb{E}_{\mathbb{Q}}[X_i\mid\mathcal{F}_{i-1}],
\qquad M:=\sum_{i=1}^n u_iX_i.
\]
Here $\{i\leq\sigma\}\in\mathcal{F}_{i-1}$, so each $u_i$ is
$\mathcal{F}_{i-1}$-measurable. Since
\[
\mathbb{E}_{\mathbb{Q}}[X_i\mid\mathcal{F}_{i-1}]
=\rho\,\mathbb{E}_{\mathbb{Q}}[Z_i\mid\mathcal{F}_{i-1}],
\]
we have $|u_i|\leq\rho$. The likelihood ratio satisfies
\[
\frac{F_{i\wedge\sigma}}{F_{(i-1)\wedge\sigma}}=1+u_iX_i.
\]
Thus
\[
\log F_\sigma=\sum_{i=1}^n\log(1+u_iX_i)\leq M,
\qquad M\geq\log\eta\quad\text{on }A.
\]
Define also
\[
S_i:=\frac{u_i(X_i-\rho Z_i)}{1-\rho Z_i\operatorname{sgn}(u_i)},
\qquad S:=\sum_{i=1}^n S_i,
\]
where $S_i=0$ when $u_i=0$.

\emph{Plan of the proof.} The main observation, obtained from
Lemma~\ref{lem:boolean-two-point}, is the comparison
\[
\mathbb{E}_{\mathbb{Q}}[1_AS]
\geq\eta\mathbb{E}_{\mu}[1_AM]
\geq\eta\log\eta\,\mu(A).
\]
The first inequality follows because the difference is a sum of
nonnegative contributions from edges crossing the level $\eta$.
Conditionally on $Z$, the $S_i$ are martingale differences. The moment
estimate below and the stopped energy bound give
\[
\mathbb{E}_{\mathbb{Q}}S=0,
\qquad
\mathbb{E}_{\mathbb{Q}}S^2
\leq\frac{1+\rho}{1-\rho}\,
\mathbb{E}_{\mathbb{Q}}\sum_i u_i^2
\leq4\frac{1+\rho}{1-\rho}\log\eta.
\]
Cauchy--Schwarz bounds the same expectation from above:
\[
\begin{aligned}
\mathbb{E}_{\mathbb{Q}}[1_AS]
&=\mathbb{E}_{\mathbb{Q}}[(1_A-\tfrac12)S]\\
&\leq\tfrac12\sqrt{\mathbb{E}_{\mathbb{Q}}S^2}
\leq\sqrt{\frac{1+\rho}{1-\rho}\log\eta}.
\end{aligned}
\]

The following lemmas hold for arbitrary
$\mathcal{F}_{i-1}$-measurable functions $u_i$, with the same notation.

\begin{lemma}\label{lem:boolean-two-point}
For every function $g$ on the cube and every $i$,
\[
\mathbb{E}_{\mathbb{Q}}[S_i g(X)\mid Z]
=
\mathbb{E}_{\mathbb{Q}}\!\left[
(u_iX_i)_+\bigl(g(X)-g(X^{(i)})\bigr)
\,\middle|\,Z\right].
\]
The same identity holds under $\mu$, with $S_i$ replaced by $u_iX_i$.
\end{lemma}

\begin{proof}
By the definition of $S_i$,
\[
S_i=(u_iX_i)_+
-\frac{1-\rho Z_iX_i}{1+\rho Z_iX_i}(u_iX_i)_-.
\]
Conditionally on $Z$, the fraction is the ratio of the probabilities
of $X^{(i)}$ and $X$. Since $u_i$ does not depend on $X_i$, changing
$X$ to $X^{(i)}$ in the second term proves the identity. Under $\mu$, this ratio is $1$ and
$S_i$ is replaced by $(u_iX_i)_+-(u_iX_i)_-=u_iX_i$.
\end{proof}

Taking $g=1_A$ in Lemma~\ref{lem:boolean-two-point}, summing over $i$,
and using $\mathbb{E}_{\mathbb{Q}}g(X)=\mathbb{E}_{\mu}[Fg]$, we obtain
\[
\begin{aligned}
&\mathbb{E}_{\mathbb{Q}}[1_AS]-\eta\mathbb{E}_{\mu}[1_AM]\\
&\qquad=
\mathbb{E}_{\mu}\sum_{i=1}^n
(u_iX_i)_+(F(X)-\eta)
\bigl(1_A(X)-1_A(X^{(i)})\bigr)\geq0.
\end{aligned}
\]
Every summand is nonnegative, since $A=\{F>\eta\}$. Thus
\begin{equation}\label{eq:boolean-comparison}
\mathbb{E}_{\mathbb{Q}}[1_AS]\geq\eta\mathbb{E}_{\mu}[1_AM].
\end{equation}

\begin{lemma}\label{lem:boolean-moments}
With the notation above,
\[
\mathbb{E}_{\mathbb{Q}}[S\mid Z]=0,
\]
and
\[
\begin{aligned}
\mathbb{E}_{\mathbb{Q}}[S^2\mid Z]
&=\mathbb{E}_{\mathbb{Q}}\left[\sum_{i=1}^n S_i^2\,\middle|\,Z\right]\\
&\leq\frac{1+\rho}{1-\rho}\,
\mathbb{E}_{\mathbb{Q}}\left[\sum_{i=1}^n u_i^2\,\middle|\,Z\right].
\end{aligned}
\]
\end{lemma}

\begin{proof}
Conditionally on $Z$ and $\mathcal{F}_{i-1}$, the coefficient $u_i$ is
fixed and $X_i$ has mean $\rho Z_i$ and variance $1-\rho^2$. Hence
\[
\mathbb{E}_{\mathbb{Q}}[S_i\mid Z,\mathcal{F}_{i-1}]=0,
\]
and
\[
\mathbb{E}_{\mathbb{Q}}[S_i^2\mid Z,\mathcal{F}_{i-1}]
=u_i^2\frac{1+\rho Z_i\operatorname{sgn}(u_i)}
                 {1-\rho Z_i\operatorname{sgn}(u_i)}
\leq\frac{1+\rho}{1-\rho}\,u_i^2.
\]
The mixed terms vanish after conditioning on $Z$, and summing proves
the result.
\end{proof}

\begin{theorem}\label{thm:boolean}
For every $\eta>1$,
\[
\mu(T_{\rho}f>\eta)
\leq
\sqrt{\frac{1+\rho}{1-\rho}}\,
\frac{1}{\eta\sqrt{\log\eta}}.
\]
\end{theorem}

\begin{proof}
By Markov's inequality, it is enough to consider $\eta\geq e$.
Use the stopped coefficients $u_i$ and $A,\sigma,M,S$ defined above.
The last likelihood increment is at most $1+\rho$, so
\[
F_\sigma\leq\eta(1+\rho).
\]
For $|a|<1$,
\[
\frac{1+a}{2}\log(1+a)+\frac{1-a}{2}\log(1-a)\geq\frac{a^2}{2}.
\]
Conditioning on $\mathcal{F}_{i-1}$ under $\mathbb{Q}$ and summing gives
\[
\mathbb{E}_{\mathbb{Q}}\sum_{i=1}^n u_i^2
\leq2\mathbb{E}_{\mathbb{Q}}\log F_{\sigma}
\leq2\log\bigl(\eta(1+\rho)\bigr)
\leq4\log\eta.
\]
By Lemma~\ref{lem:boolean-moments},
\[
\mathbb{E}_{\mathbb{Q}}S=0,
\qquad
\mathbb{E}_{\mathbb{Q}}S^2
\leq4\frac{1+\rho}{1-\rho}\log\eta.
\]
Finally, \eqref{eq:boolean-comparison} and Cauchy--Schwarz give
\[
\begin{aligned}
\eta\log\eta\,\mu(A)
&\leq\eta\mathbb{E}_{\mu}[1_AM]
\leq\mathbb{E}_{\mathbb{Q}}[1_AS]\\
&=\mathbb{E}_{\mathbb{Q}}[(1_A-\tfrac12)S]
\leq\tfrac12\sqrt{\mathbb{E}_{\mathbb{Q}}S^2}
\leq\sqrt{\frac{1+\rho}{1-\rho}\log\eta}.
\end{aligned}
\]
Dividing by $\eta\log\eta$ proves the result.
\end{proof}

\section{The role of AI in this proof}
The author acknowledges the use of ChatGPT-5.6 assistance to close the argument
based on the supplied plan and preliminary gaussian case draft developed by the author. 
Furthermore,  ChatGPT was used for proof-reading and numerical experiments.

\end{document}